\documentclass[11pt,a4paper]{article}
\usepackage[T1]{fontenc}
\usepackage{lmodern}
\usepackage[margin=27mm]{geometry}
\usepackage{amsmath,amssymb,amsthm,microtype}
\usepackage{enumitem}
\usepackage{hyperref}
\hypersetup{colorlinks=true,linkcolor=blue,citecolor=blue,urlcolor=blue,
 pdftitle={A Countable Three-Generated Orthomodular Lattice with No Sigma-Complete Embedding},
 pdfauthor={Joaquim Reizi Higuchi}}
\setlist{itemsep=2pt,topsep=4pt}
\numberwithin{equation}{section}
\newtheorem{theorem}{Theorem}[section]
\newtheorem{lemma}[theorem]{Lemma}
\newtheorem{remark}[theorem]{Remark}
\newcommand{\R}{\mathbb R}
\newcommand{\N}{\mathbb N}
\newcommand{\spn}{\operatorname{span}}
\newcommand{\gr}{\operatorname{graph}}
\newcommand{\id}{\operatorname{id}}
\newcommand{\cc}{\mathrel{\mathsf C}}
\title{A Countable Three-Generated Orthomodular Lattice\\
with No $\sigma$-Complete Embedding}
\author{Joaquim Reizi Higuchi\\[4pt]
\small Graduate School, Program in Natural Sciences\\
\small The Open University of Japan, Chiba, Japan\\
\small \href{mailto:joaquim_reizi@icloud.com}{\texttt{joaquim\_reizi@icloud.com}}\\
\small ORCID: \href{https://orcid.org/0009-0006-1392-3770}{0009-0006-1392-3770}}
\date{10 October 2026\\[3pt]\small Preprint draft, version 3}

\begin{document}
\maketitle
\begin{abstract}
We give an explicit countably infinite orthomodular lattice generated by
three specified elements and a direct argument that it admits no injective
homomorphism into any $\sigma$-complete orthomodular lattice. In particular,
it does not embed into any set-sized complete orthomodular lattice.
The source is a generated subalgebra of a cofinite reduced product of
finite-dimensional real subspace lattices. A nonzero finite term $a$ is
orthogonal to the orbit of a terminal line under every finite word in three
Sasaki projections, while another finite term $X$ satisfies
$\pi_a\pi_X(t)=a$. If the image of this countable orbit has a supremum in a
target orthomodular lattice, Sasaki adjunction forces that supremum to
commute with the generator images. This forces the image of $a$ to vanish.
The proof imposes no atomicity, modularity, Hilbert representation,
cardinality bound on the target, or preservation of source infinitary
operations. All target lemmas are derived from the orthomodular axioms.
\end{abstract}

\begin{center}
\small\textbf{Verification scope.} The paper argument has received internal
AI paper reconstruction. It has not received formal verification or
independent external review. Novelty and priority remain unestablished.
\end{center}

\section{Introduction}
The question whether every orthomodular lattice embeds into a complete
orthomodular lattice is Problem~1 in Harding and
Wang~\cite[Section~3]{HardingWang2021v1}; their Problem~2 asks about
$\sigma$-complete targets. An embedding here is an injective map preserving
the bounded ortholattice signature
\[
                        0,\quad1,\quad\wedge,\quad\vee,\quad{}'.
\]
It need not preserve infinite joins already present in its source.

\begin{theorem}[Main result]\label{thm:main}
In ZFC there is a set-sized, countably infinite orthomodular lattice $L$,
generated by three specified elements, with no embedding into any
set-sized $\sigma$-complete orthomodular lattice. Consequently, $L$ has
no embedding into any set-sized complete orthomodular lattice.
\end{theorem}

Three-generated means generated by the displayed triple; no minimal
generator cardinality is asserted. We prove a stronger local statement:
there are $a\ne0$ and a countable $O\subseteq L$ such that every
homomorphism $f:L\to C$ into any orthomodular lattice satisfies $f(a)=0$
whenever the single supremum $\bigvee_C f[O]$ exists.

We first derive the finite lemmas needed in arbitrary targets, then
isolate an orbit obstruction and construct an actual source satisfying
it. The proof is self-contained apart from finite-dimensional real
linear algebra and ordinary set constructions.
Section~\ref{sec:literature} treats the literature comparison separately.
Section~\ref{sec:status} records verification scope and a distinct
historical repair-congruence question.

\section{Definitions}
An \emph{orthomodular lattice} (OML) is a bounded lattice with an
order-reversing involution $x\mapsto x'$ satisfying
$x\wedge x'=0$, $x\vee x'=1$, and
\begin{equation}\label{eq:om}
            x\le y\quad\Longrightarrow\quad y=x\vee(y\wedge x').
\end{equation}
De Morgan identities follow from the order-reversing involution.
Elements $x,y$ are orthogonal when $x\le y'$, equivalently $y\le x'$.
A homomorphism preserves the five operations above and therefore all
finite terms, by structural induction. It is order-preserving because
$x\le y$ means $x=x\wedge y$.

We call an OML \emph{$\sigma$-complete} if every countable subset has a
supremum. Lemma~\ref{lem:orthogonal-sigma} also covers the convention
requiring only countable pairwise orthogonal suprema.
A \emph{complete} OML has suprema and infima of all subsets.
All carriers are sets; no relationship between the cardinalities of
source and target is prescribed.

The \emph{Sasaki projection} is the finite term
$\pi_b(x)=b\wedge(b'\vee x)$. Composition is read from right to left.
A word $w=s_k\cdots s_1$ on parameters $S$ acts as
$w(x)=\pi_{s_k}\cdots\pi_{s_1}(x)$; the empty word acts identically.

\section{Finite lemmas in an arbitrary OML}\label{sec:finite}

\begin{lemma}\label{lem:dual}
In any OML,
\begin{align}
 z\le b&\ \Longrightarrow\ b\wedge(b'\vee z)=z,\label{eq:dual}\\
 b'\vee x&=b'\vee\pi_b(x).\label{eq:E}
\end{align}
\end{lemma}
\begin{proof}
If $z\le b$, apply~\eqref{eq:om} to $b'\le z'$ to obtain
$z'=b'\vee(z'\wedge b)$. Taking complements gives~\eqref{eq:dual}.
Applying~\eqref{eq:om} to $b'\le b'\vee x$ gives
$b'\vee x=b'\vee((b'\vee x)\wedge b)=b'\vee\pi_b(x)$.
\end{proof}

\begin{lemma}[Adjunction and existing suprema]\label{lem:adjunction}
For arbitrary $b,x,y$ in an OML,
\begin{equation}\label{eq:adjunction}
       \pi_b(x)\le y\quad\Longleftrightarrow\quad
       x\le b'\vee(b\wedge y).
\end{equation}
If $Y$ has an actual supremum $j$, then $\pi_b[Y]$ has a supremum and
\begin{equation}\label{eq:sup}
                       \pi_b(j)=\bigvee_{y\in Y}\pi_b(y).
\end{equation}
\end{lemma}
\begin{proof}
If $\pi_b(x)\le y$, then $\pi_b(x)\le b\wedge y$.
Equation~\eqref{eq:E} gives
$x\le b'\vee x=b'\vee\pi_b(x)\le b'\vee(b\wedge y)$.
Conversely, $x\le b'\vee(b\wedge y)$ and monotonicity give
\[
 \pi_b(x)\le b\wedge\bigl(b'\vee(b\wedge y)\bigr)=b\wedge y\le y,
\]
where~\eqref{eq:dual} applies to $b\wedge y\le b$.
Monotonicity makes $\pi_b(j)$ an upper bound for the image family.
Any upper bound $v$ of that family, by adjunction, makes
$b'\vee(b\wedge v)$ an upper bound for $Y$.
The leastness of $j$ and adjunction give $\pi_b(j)\le v$.
This proves existence and leastness of the image supremum without
an infinite distributive law.
\end{proof}

\begin{lemma}[Projection stability and compatibility]\label{lem:compatibility}
If $\pi_s(c)\le c$, then
\begin{equation}\label{eq:compat}
 c=(c\wedge s)\vee(c\wedge s'),\qquad
 s=(s\wedge c)\vee(s\wedge c').
\end{equation}
Conversely, the first decomposition gives $\pi_s(c)=s\wedge c$.
\end{lemma}
\begin{proof}
Put $u=\pi_s(c)$. We have $s\wedge c\le u$, while the hypothesis
gives $u\le s\wedge c$; hence $u=s\wedge c$.
Adjunction gives $c\le s'\vee u$.
Orthomodularity gives $c=u\vee v$, where $v=c\wedge u'$.
Since $s'\le u'$, equation~\eqref{eq:dual} gives
$v\le u'\wedge(u\vee s')=s'$.
The reverse inequality $c\wedge s'\le c\wedge u'=v$ gives
$v=c\wedge s'$.

Write $s=u\vee w$ with $w=s\wedge u'$.
Since $v\le s'$, we have $w\le s\le v'$ and thus
$w\le u'\wedge v'=(u\vee v)'=c'$.
Also $s\wedge c'\le s\wedge u'=w$, so $w=s\wedge c'$.
This proves symmetry.
For the converse, the first decomposition gives
$s'\vee c=s'\vee(s\wedge c)$, and~\eqref{eq:dual} gives
$\pi_s(c)=s\wedge c$.
\end{proof}

Write $x\cc c$ for these symmetric compatibility decompositions.

\begin{lemma}[Finite closure of a commutant]\label{lem:commutant}
For fixed $c$, the elements compatible with $c$ form a bounded
subOML. If $x\cc c$, then $\pi_x(c)=x\wedge c$.
\end{lemma}
\begin{proof}
Suppose $x=u\vee v$, $u\le c$, $v\le c'$.
Since $c\le v'$, equation~\eqref{eq:dual} gives
$x\wedge c\le v'\wedge(v\vee u)=u$; the reverse inclusion follows
from $u\le x,c$. Similarly $c'\le u'$ gives
$x\wedge c'\le u'\wedge(u\vee v)=v$, and $v\le x,c'$ gives the reverse.
Thus the decomposition is exactly compatibility under $c$.
It yields $\pi_c(x)=c\wedge x\le x$ and hence its symmetric form by
Lemma~\ref{lem:compatibility}.

For two such decompositions, finite join laws give
$x_1\vee x_2=(u_1\vee u_2)\vee(v_1\vee v_2)$ with the same corner
bounds. No distribution is used.
For the complement put $B=(c\wedge u')\vee(c'\wedge v')$.
Each term is below both $u'$ and $v'$, since $c\le v'$ and $c'\le u'$.
Hence $B\le x'$. Orthomodularity applied to $u\le c$ and $v\le c'$
gives $x\vee B=c\vee c'=1$.
Applied to $B\le x'$, it gives
$x'=B\vee(x'\wedge B')=B\vee(x\vee B)'=B$.
Thus complements have the required decomposition.
Meet closure follows by De Morgan; $0=0\vee0$ and $1=c\vee c'$ are
included. The subalgebra inherits orthomodularity.
Finally $c=(x\wedge c)\vee(x'\wedge c)$ gives
$x'\vee c=x'\vee(x\wedge c)$.
Equation~\eqref{eq:dual} then gives $\pi_x(c)=x\wedge c$.
\end{proof}

\section{A countable orbit obstruction}\label{sec:obstruction}

\begin{theorem}[Orbit obstruction]\label{thm:obstruction}
Let $L$ be generated as a bounded OML by a finite or countable set $S$.
Suppose $a,t,X\in L$ satisfy
\begin{equation}\label{eq:source-identities}
 a\ne0,\qquad \pi_a\pi_X(t)=a,\qquad
 w(t)\le a'\quad(w\in S^*).
\end{equation}
Put $O=\{w(t):w\in S^*\}$. For every homomorphism $f:L\to C$
into an arbitrary OML, existence of $c=\bigvee_C f[O]$ implies $f(a)=0$.
No such map into a $\sigma$-complete OML is injective.
\end{theorem}
\begin{proof}
The orbit is countable, contains $t$, and satisfies $\pi_s[O]\subseteq O$.
Lemma~\ref{lem:adjunction} and finite-term preservation give
\[
                 \pi_{f(s)}(c)=\bigvee_{o\in O}f(\pi_s(o))\le c.
\]
Thus $c$ is compatible with every $f(s)$.
Every element of an algebra generated under finitary operations is a
finite term in its generators. Lemma~\ref{lem:commutant} implies $f(X)\cc c$.
The orbit inequalities give $f(o)\le f(a)'$, so target leastness gives
$c\le f(a)'$. The empty word gives $f(t)\le c$.
Set $\alpha=f(a)$ and $\xi=f(X)$. Then
\begin{align*}
 \alpha&=\pi_\alpha\pi_\xi(f(t))
 \le\pi_\alpha\pi_\xi(c)
 =\pi_\alpha(\xi\wedge c)\\
 &\le\pi_\alpha(c)
 =\alpha\wedge(\alpha'\vee c)=\alpha\wedge\alpha'=0.
\end{align*}
The inequalities use monotonicity; the middle equality uses established
compatibility. Thus $f(a)=f(0)$ although $a\ne0$.
Countability of $O$ supplies $c$ in a $\sigma$-complete target.
\end{proof}

\begin{remark}
The supremum is in the target; no source supremum or infinitary
preservation by $f$ is assumed. Compatibility with $f[L]$ suffices;
centrality of $c$ in all of $C$ is not asserted.
\end{remark}

\section{The explicit source}\label{sec:source}
\subsection{An actual cofinite quotient}

For $N\ge1$, take $H_N=\R^N$ with orthonormal basis
$e_0,\ldots,e_{N-1}$ and $E_N=H_N\oplus H_N$. Define
\[
 T_Ne_i=e_{i+1}\ (i<N-1),\qquad T_Ne_{N-1}=0.
\]
Let $M_N$ be the lattice of all real linear subspaces of $E_N$, with intersection,
linear sum, and Euclidean orthogonal complement.
It is a set. Arbitrary intersections and spans of unions give complete
lattice operations. If $A\subseteq B$, orthogonal projection decomposes
each $b\in B$ as $P_Ab+(b-P_Ab)$, with the second term in $B\cap A^\perp$.
Thus $B=A+(B\cap A^\perp)$ and $M_N$ is a complete OML.
Finite dimension does not mean a finite carrier.

Give the set-sized product $M=\prod_{N\ge1}M_N$ its coordinatewise
bounded ortholattice operations. Define $A\sim B$ by
$A_N=B_N$ for all sufficiently large $N$.
It is a congruence for finite operations: unary operations use the same
equality tail, and binary operations use an intersection of two tails.
The quotient $M_F=M/{\sim}$ has order
\begin{equation}\label{eq:quotient-order}
 [A]\le[B]\quad\Longleftrightarrow\quad
 A_N\subseteq B_N\text{ for all sufficiently large }N.
\end{equation}
Both sides mean $[A\wedge B]=[A]$.
On the tail of this order comparison, coordinate orthomodularity gives
$[B]=[A]\vee([B]\wedge[A]')$.
Bounded ortholattice identities also pass to the quotient.
The bounds differ at every coordinate, so $0\ne1$.
Hence $M_F$ is an actual set-sized OML; its completeness is not assumed.

Let $p,q,r$ be the classes of
\begin{equation}\label{eq:generators}
 p_N=H_N\oplus0,\qquad q_N=\gr T_N,\qquad r_N=\gr\id_{H_N}.
\end{equation}
Let $L=L_{3F}$ be the bounded subOML they generate.
It inherits the order and orthomodularity, since the operations in
\eqref{eq:om} remain inside it.
There are countably many finite terms over this finite alphabet,
so $L$ is at most countable. It is an actual subalgebra, not an
unchecked free presentation.

\subsection{Endpoint and global lines as finite terms}

For Euclidean subspaces, the Sasaki term is the orthogonal image:
\begin{equation}\label{eq:projection-image}
                         \pi_S(V)=P_S[V].
\end{equation}
If $z\in S\cap(S^\perp+V)$, write $z=u+v$ with $u\in S^\perp$, $v\in V$;
then $z=P_Sv$. Conversely $P_Sv=v-P_{S^\perp}v$ is in both
$S$ and $S^\perp+V$.
Define the finite terms
\begin{equation}\label{eq:a-t-X}
 a=\pi_p\pi_r(p'\wedge q'),\qquad t=p\wedge q,\qquad
 X=\pi_p\bigl(((q\wedge t')\vee r)'\bigr).
\end{equation}

\begin{lemma}\label{lem:coordinates}
For all $N\ge1$, with $u_N=\sum_{i<N}e_i$,
\[
 a_N=\R(e_0,0),\qquad t_N=\R(e_{N-1},0),\qquad X_N=\R(u_N,0).
\]
Therefore $\pi_a\pi_X(t)=a\ne0$ in $L$.
\end{lemma}
\begin{proof}
An inner-product calculation gives
$q_N^\perp=\{(-T_N^*y,y):y\in H_N\}$.
Its intersection with $p_N^\perp$ forces $T_N^*y=0$;
the kernel is $\R e_0$, so the intersection is $\R(0,e_0)$.
Projection to $r_N$ sends $(0,e_0)$ to $(e_0/2,e_0/2)$,
then projection to $p_N$ gives $(e_0/2,0)$.
Thus $a_N=\R(e_0,0)$.
The kernel $\ker T_N=\R e_{N-1}$ gives $t_N=p_N\cap q_N=\R(e_{N-1},0)$.

The graph vectors orthogonal to $t_N$ are spanned by
$(e_i,e_{i+1})$ for $i<N-1$.
A vector orthogonal to $r_N$ is $(x,-x)$.
Orthogonality to those graph vectors says $x_i=x_{i+1}$.
Thus $((q_N\cap t_N^\perp)+r_N)^\perp=\R(u_N,-u_N)$,
whose projection to $p_N$ is $\R(u_N,0)$.
For $N=1$ the adjacency conditions are empty, $T_1=0$,
and $a_1=t_1=X_1=p_1=q_1$, so the formulas remain valid.

Since $\|u_N\|^2=N$ and $\langle e_{N-1},u_N\rangle=1$,
\begin{equation}\label{eq:scalar}
               P_{a_N}P_{X_N}(e_{N-1},0)=(e_0/N,0).
\end{equation}
The coefficient is nonzero, so the image of the terminal line is
exactly $a_N$. Equation~\eqref{eq:projection-image} gives the Sasaki
identity at every coordinate and hence in the quotient.
The line $a_N$ is nonzero at every coordinate, so $a\ne0$.
\end{proof}

\begin{remark}
Equation~\eqref{eq:scalar} identifies a subspace image. The quotient uses
eventual equality of subspaces, not a norm limit; $1/N$ does not
make the quotient line zero.
\end{remark}

\begin{lemma}\label{lem:infinite}
The source $L$ is countably infinite.
\end{lemma}
\begin{proof}
Put $F(x)=\pi_p\pi_r\pi_{p'}\pi_q(x)$.
For $i<N-1$, the successive images of $\R(e_i,0)$ are
\[
 \R(e_i,e_{i+1}),\quad\R(0,e_{i+1}),\quad
 \R(e_{i+1},e_{i+1}),\quad\R(e_{i+1},0).
\]
All intermediate coefficients are nonzero. The terminal first-axis
line is fixed by $\pi_q$ and killed by $\pi_{p'}$.
Hence $F^n(a)_N=\R(e_n,0)$ for $N>n$, and is zero for $N\le n$.
For fixed $m\ne n$, these lines differ for every $N>\max(m,n)$,
so their classes differ. All belong to the already countable $L$.
\end{proof}

\subsection{All finite words and their support}
Put $S=\{p,q,r\}$ and $O=\{w(t):w\in S^*\}$, and set
\[
 K_N(\ell)=\spn\{(e_i,0),(0,e_i):\ell\le i<N\}\quad(0\le\ell\le N).
\]
Coordinate word values use the corresponding parameters $p_N,q_N,r_N$
in the same fixed word.
We have $t_N\subseteq K_N(N-1)$.
The projections
$P_{p_N}(x,y)=(x,0)$ and
$P_{r_N}(x,y)=((x+y)/2,(x+y)/2)$ preserve every $K_N(\ell)$.
Writing $P_{q_N}(x,y)=(z,T_Nz)$, orthogonality of the residual to
the graph gives $(I+T_N^*T_N)z=x+T_N^*y$.
The inverse has diagonal entries $1/2,\ldots,1/2,1$, so
\begin{equation}\label{eq:qprojection}
 z_i=(x_i+y_{i+1})/2\ (i<N-1),\qquad z_{N-1}=x_{N-1}.
\end{equation}
The second component is zero at index $0$ and equals $z_{j-1}$ at $j\ge1$.
If $(x,y)\in K_N(\ell)$, then $z_i=0$ for $i<\ell-1$,
and $T_Nz$ has no support below that bound. Therefore
$P_{q_N}[K_N(\ell)]\subseteq K_N(\max(0,\ell-1))$.
This is a statement for every vector, including $\ell=0$ and $N=1$.

\begin{lemma}\label{lem:allwords}
For every word $w\in S^*$ of length $k$ and every $N\ge1$,
$w(t_N)\subseteq K_N(\max(0,N-1-k))$.
Consequently $w(t)\le a'$ for every fixed $w$.
\end{lemma}
\begin{proof}
The empty word has value $t_N$.
An outer $p$ or $r$ preserves the current lower support bound,
and an outer $q$ lowers it by at most one.
Induction on length proves the inclusion for all $N$.
If $N\ge k+2$, the lower support bound is at least $1$, so the
whole subspace is orthogonal to $a_N=\R(e_0,0)$.
Equation~\eqref{eq:quotient-order} now gives $w(t)\le a'$.
\end{proof}

The quantifiers are $\forall w\,\exists N_w\,\forall N\ge N_w$,
with $N_w=|w|+2$. No common tail or countable intersection of
cofinite tails is used. The separate infinitude term $F$ uses $p'$,
but $p'$ is not added to this orbit's alphabet.

\section{Nonembedding and the \texorpdfstring{$\sigma$}{sigma} conventions}
\begin{proof}[Proof of Theorem~\ref{thm:main}]
The constructed $L$ is an actual set-sized, countably infinite
OML generated by $p,q,r$.
Lemmas~\ref{lem:coordinates} and~\ref{lem:allwords} discharge all source
assumptions of Theorem~\ref{thm:obstruction}.
Every homomorphism into a $\sigma$-complete OML therefore kills $a\ne0$,
so is not injective. Every complete OML supplies all countable suprema,
which gives the complete-target conclusion.
\end{proof}

\begin{lemma}\label{lem:orthogonal-sigma}
Countable pairwise orthogonal suprema in an OML imply all countable suprema.
\end{lemma}
\begin{proof}
Given $x_0,x_1,\ldots$, put $c_n=x_0\vee\cdots\vee x_n$,
$d_0=c_0$, and $d_n=c_n\wedge c_{n-1}'$ for $n\ge1$.
Orthomodularity gives $c_n=c_{n-1}\vee d_n$ and, inductively,
$c_n=\bigvee_{i\le n}d_i$.
For $m<n$, $d_m\le c_{n-1}$ and $d_n\le c_{n-1}'$, so the $d_n$
are pairwise orthogonal.
Let $d=\bigvee_n d_n$. Every $c_n$ and $x_n$ lies below $d$.
Every upper bound of the $x_n$ bounds all $c_n$ and $d_n$, hence is
above $d$. Thus $d$ is the original sequence's supremum.
Zeros may be removed if the orthogonal-family convention excludes them.
Empty and finite cases use $0$ and finite joins.
\end{proof}

The target lemmas used only the OML axioms and actual leastness in the
target. Euclidean geometry was used solely to construct the source.
Target atomicity, modularity, Hilbert representation, and a target
cardinality bound are not assumptions.

\section{Relation to the cited literature}\label{sec:literature}
The cited survey's problem statements refer to version~1 dated
22 August 2021~\cite{HardingWang2021v1}. They are historical problem
statements, not an exhaustive survey through the present date.
No priority claim is made here.

Harding~\cite[Theorem~6.1]{Harding2002} proves that every countable
OML embeds into a three-generated OML. The generating-set size alone
is therefore not the distinctive feature of this argument; the
explicit terms and projection-word obstruction are essential.

Bruns and Harding~\cite[Theorem~5]{BrunsHarding1998} prove strong
Boolean amalgamation for OMLs, discussed in Theorem~3.17 and
Remark~3.18 of the cited survey version.
The theorem supplies finite-operation embeddings. Its strong
intersection condition is not arbitrary-supremum preservation.
The remark discusses amalgamation with a Boolean completion of a block
and explicitly notes the difficulty of preserving earlier joins
during iteration.

A Boolean subalgebra complete in its own induced order need not retain
its infinite joins in the ambient algebra. With a nonprincipal ultrafilter $U$ on
$\N$, define
\[
 h(A)=A\cup\{\infty\}\quad(A\in U),\qquad h(A)=A\quad(A\notin U).
\]
This is an injective bounded Boolean homomorphism from $\mathcal P(\N)$
to $\mathcal P(\N\cup\{\infty\})$.
Indeed, ultrafilter membership preserves finite intersections, finite
unions, and complements; $h(A)\cap\N=A$ proves injectivity.
The ambient supremum of $h(\{n\})=\{n\}$ is $\N$, whereas
$h(\N)=\N\cup\{\infty\}$.
The example distinguishes internal from ambient leastness; it is
not claimed to be the cited construction's particular amalgam.

The original amalgamation proof enlarges the components and their
common Boolean algebra before constructing its final normal ideals.
The stated theorem supplies no regularity bridge for the original
completed Boolean component. We do not assert that the final
normal-ideal embedding loses joins already existing in an enlarged
component, and do not assign a precise universal quantification to
the survey's informal final comment about inserting a particular join.
The obstruction uses the actual ambient supremum of $f[O]$,
which a Boolean component's internal join alone does not supply.

\section{Verification, provenance, and separate questions}\label{sec:status}
\paragraph{Internal paper reconstruction.}
The manuscript was assembled from the saved source construction and
expanded orbit obstruction. Separate executions in the same AI
environment reconstructed the source calculations and target algebra.
A full reconstruction also checked all source assumptions, the target
lemmas, both $\sigma$ conventions, and the expanded finite-term argument,
reporting no required mathematical repair. These are provenance labels,
not premises of the proof. The English manuscript also receives a check
against those source documents.

\paragraph{Formal and external verification; novelty.}
No Lean or other proof assistant verification has been performed.
No authenticated independent external review has been obtained.
Internal checking does not establish novelty or priority, and the
limited source comparison is not an exhaustive literature survey.
These questions are separate from the explicit paper derivation.

\paragraph{Historical repair congruence.}
An earlier positive-completion route studied the congruence
\[
                 \rho=\operatorname{Cg}_{Q_{\mathrm{pre}}}(d_0,0)
\]
on a separately constructed precursor and its restriction to an earlier source.
The full restriction remains unresolved in the saved research record.
Neither the precursor, the congruence, nor a free $\sigma$-envelope is
used in Theorem~\ref{thm:main}.
The conclusion $f(a)=0$ does not determine that congruence, whose
quotient is not known to supply the required orbit supremum.
No historical congruence conclusion is asserted here.

\paragraph{Alternative quotients.}
The fixed-word cofinite relations also pass to every nonprincipal
ultraproduct of the factors, so the same argument applies individually
to the corresponding generated subOML.
No isomorphism between the cofinite and ultraproduct sources is required
or asserted. Distinct identities of designated generators alone would
not establish nonisomorphism of the unlabelled lattices.

\paragraph{AI-use disclosure.}
The author had obtained the conclusion of this paper before the present
use of AI. Codex was used to check the theory and to prepare and
proofread the manuscript. Python was used for computational checks
over a finite range.

\appendix
\section{Bounded computation and reproducibility}
The proof is the all-$N$ coordinate calculation and induction on every
fixed word. A supplementary exact rational calculation used generic
subspace meet, join, orthogonal complement, and Sasaki operations with
Fraction/RREF arithmetic for $1\le N\le8$ and all $p,q,r$ words of
lengths $0$ through $5$.
It checked $8(1+3+3^2+3^3+3^4+3^5)=2912$ word instances, the line
identity, its coefficient $1/N$, and endpoint cases of the infinitude
term, with no reported errors. This finite check is not the proof
of the universal statements.

The local research archive includes the manuscript, bibliographic
identifier validation and version notes, internal manuscript checking,
build and layout records, source hashes, and underlying proof documents.
The archive distinguishes local preparation from any subsequent
arXiv submission and contains no formal-verification certificate.
Identifier validation establishes cited-record
existence, not the correctness of theorem application.

\end{document}